\documentclass[11pt,letterpaper,reqno]{amsart}
\usepackage[margin=1.45in]{geometry}
\usepackage{hyperref}
\newtheorem{theorem}{Theorem}
\newtheorem{lemma}[theorem]{Lemma}

\newtheorem{corollary}[theorem]{Corollary}
\theoremstyle{remark}

\newcommand{\n}[2][]{#1\lVert #2 #1\rVert}
\newcommand{\abs}[2][]{#1\lvert #2 #1\rvert}
\newcommand{\pp}[2]{\frac{\partial #1}{\partial #2}}
\newcommand{\R}{\mathbb{R}}
\newcommand{\without}{\setminus}
\newcommand{\asa}{\textup{~as~}}
\newcommand{\ina}{\textup{~in~}}
\newcommand{\ona}{\textup{~on~}}
\newcommand{\maps}{\colon}
\newcommand{\nin}{\notin}
\newcommand{\by}{\times}
\newcommand{\sub}{\subset}
\newcommand{\inv}{^{-1}}
\newcommand{\dell}{\partial}
\newcommand{\grad}{\nabla}

\newcommand\surf{S}                
\newcommand\fluid{\Omega}          
\newcommand\sdim{n}                
\newcommand\bx{x}                  
\newcommand\nx{x'}                 
\newcommand\dS{dS}                 
\newcommand\ny{x_n}                
\newcommand\bc{c}                  
\newcommand\bcx{c'}                
\newcommand\ba{p}                  
\newcommand\bay{p_n}               
\newcommand\bax{p'}                
\newcommand\bn{N}                  
\newcommand\ey{{e_n}}              
\newcommand\bA{A}                  
\newcommand\bC{\tilde \bA}         
\newcommand\bnu{\nu'}              
\newcommand{\kel}[1]{\tilde #1}    
\newcommand{\kfluid}{\Omega^\sim}  
\newcommand{\kphi}{\kel\varphi}    
\newcommand{\kbx}{{\kel x}}        
\newcommand{\knx}{{\kel x'}}       
\newcommand{\kny}{{\kel x_n}}      
\newcommand{\kax}{{x^*}}           
\newcommand{\ksurf}{S^\sim}        
\newcommand{\kbn}{\kel\bn}         
\newcommand{\source}{\beta}        

\begin{document}

\title[Solitary waves in deep water]
{Integral and asymptotic properties of solitary waves in deep water}
\author{Miles H.~Wheeler}
\date{April 1, 2016}

\begin{abstract}
  We consider two- and three-dimensional gravity and gravity-capillary
  solitary water waves in infinite depth. Assuming algebraic decay
  rates for the free surface and velocity potential, we show that the
  velocity potential necessarily behaves like a dipole at infinity and
  obtain a related asymptotic formula for the free surface. We then
  prove an identity relating the ``dipole moment'' to the kinetic
  energy. This implies that the leading-order terms in the asymptotics
  are nonvanishing and in particular that the angular momentum is infinite.
  Lastly we prove a related integral identity which rules out waves of
  pure elevation or pure depression.
\end{abstract}
\maketitle

\section{Introduction}\label{sec:intro}

We consider the motion of an infinitely deep region fluid under the
influence of gravity which is bounded above by a free surface,
including both \emph{gravity waves}, where the pressure is constant
along the free surface, and \emph{gravity-capillary waves}, where it
is proportional to the mean curvature. The fluid is assumed to be
inviscid and incompressible, and the flow is assumed to be
irrotational. We further restrict to traveling-wave solutions which
appear steady in a moving reference frame and which are
\emph{solitary} in that the free surface approaches some asymptotic
height at infinity, normalized to zero. 

Solitary waves in finite depth, where the fluid is instead bounded
below by a flat bed, have a long and celebrated history; see for
instance the reviews \cite{miles:survey,dk:review,groves:survey}. This
includes a wide variety of existence results for gravity waves in two
dimensions \cite{lavrentiev, fh, beale, mielke, at:finite} and
gravity-capillary waves in two \cite{ak:tension, kirch:res, ik:normal,
bg:mult, bgt:plethora} and three
\cite{gs:localized,bgsw:localized,bgw:weak} dimensions. The
two-dimensional gravity waves are \emph{waves of elevation} in that
their free surface elevations are everywhere positive \cite{cs:sym},
while some of the gravity-capillary waves are \emph{waves of
depression} with negative free surfaces and still others have
oscillatory free surfaces which change sign. For the remaining case of
three-dimensional gravity waves, Craig \cite{craig:nonexistence} has
ruled out waves of elevation or depression, while recent results for
the time-dependent problem \cite{wang:onflat, wang:flat} rule out
sufficiently small solitary waves.

In infinite depth, gravity solitary waves are conjectured not to
exist, regardless of the dimension. Hur~\cite{hur:nosolitary} has
proved that the only two-dimensional solitary waves whose free
surfaces are $O(1/\abs \bx^{1+\varepsilon})$ as $\abs \bx \to \infty$
are \emph{trivial waves} with flat free surfaces. For
three-dimensional waves, Craig has shown, as in the finite-depth case,
that there are no waves of elevation or depression
\cite{craig:nonexistence}. Also as in the case of finite depth,
sufficiently small three-dimensional gravity waves are ruled out by
global existence results \cite{wu:global3d, gms:global3d} for small
data in the time-dependent problem.

Two-dimensional gravity-capillary waves in infinite depth were first
rigorously constructed by Iooss and Kirrmann \cite{ik:capdeep}
following the pioneering numerical work of Longuet-Higgins
\cite{lh:cgdeep} and Vanden-Broeck and Dias \cite{vbd:deep}. Their
proof used normal form techniques; Buffoni \cite{buffoni:deep} and
Groves and Wahl\'en \cite{gw:deep} have subsequently given variational
constructions. One distinguishing feature of these solitary waves is
their algebraic decay at infinity. In \cite{ik:capdeep} the free
surfaces were shown to be $O(1/\abs\bx)$ as $\bx \to \infty$; Sun
\cite{sun:properties} later improved this to the expected $O(1/\abs
x^2)$. More generally, Sun proved that a $O(1/\abs
\bx^{1+\varepsilon})$ free surface is automatically $O(1/\abs \bx^2)$,
and that in this case several integral identities hold. In particular,
the ``excess mass'' vanishes so that no such wave can be a wave of
elevation or depression. 

While there are no rigorous constructions of three-dimensional gravity
waves in infinite depth, they have been calculated both formally
\cite{ka:lumps} and numerically \cite{pvbc:deep,wm:deep,am:model}.
Interestingly, as the amplitude of these waves approaches zero, their
energy approaches a finite value \cite[Section~3.2.2]{wm:deep}. This
is consistent with recent global existence results for small data in
the time-dependent problem \cite{dipp:global}, which rule out solitary
waves that are small in a certain function space. 

In this paper we simultaneously consider infinite-depth gravity and
gravity-capillary solitary waves in dimension $n=2$ or $3$. We assume
that the free surface is $O(1/\abs\bx^{\sdim+\varepsilon})$ as
$\abs\bx \to \infty$ while the velocity potential is
$o(1/\abs\bx^{n-2})$. Our first conclusion is that the velocity
potential behaves like a dipole near infinity \eqref{eqn:dipole},
which implies related asymptotics \eqref{eqn:asym} for the free
surface. We next give an explicit formula \eqref{eqn:kinetic} for the
kinetic energy in terms of the ``dipole moment'' and the wave speed.
For nontrivial waves, this ensures that the leading-order terms in our
asymptotics are nonvanishing, which in turn implies that the angular
momentum is infinite (Corollary~\ref{cor:angular}). A modification of
the proof of \eqref{eqn:kinetic} shows that the ``excess mass''
vanishes \eqref{eqn:mass}, ruling out waves of elevation or depression.

We now briefly interpret our results in the context of the previous
work mentioned above. 
The two-dimensional gravity waves we consider are automatically
trivial by Hur's nonexistence result \cite{hur:nosolitary}, so that
our results are interesting only in their method of proof. 
For three-dimensional gravity waves, on the other hand, our results are
entirely new. Our nonexistence proof for waves with finite
angular momentum complements Craig's nonexistence result
\cite{craig:nonexistence}, which only applies to waves of elevation
and depression, as well as the time-dependent results
\cite{wu:global3d,gms:global3d}, which only rule out waves which are
sufficiently small. 
%
For two-dimensional gravity-capillary waves, we improve upon Sun's
asymptotic bounds \cite{sun:properties} by proving asymptotic formulas
for both the free surface and velocity potential, with nonvanishing
leading-order terms. Given this dipole-like behavior of the velocity
potential, somewhat formal proofs of our integral identities were
given in this case by Longuet-Higgens~\cite{lh:cgdeep}. 
To our knowledge, the only rigorous results for three-dimensional
solitary capillary-gravity waves in infinite depth are Hur's recent
generalization \cite{hur:energies} of one of Sun's two-dimensional
identities and the implicit nonexistence result for small waves in
\cite{dipp:global}. 

The two-dimensional results \cite{sun:properties, hur:nosolitary} both
use conformal mappings to obtain a problem in a fixed domain. Hur goes
on to write a non-local Babenko-type equation for the free surface
elevation as a function of the velocity potential, while Sun exploits
the existence of explicit Greens functions. In three dimensions, these
arguments break down entirely. Instead of conformal mappings, we use
the Kelvin transform. This does not fix the domain, but does convert
questions about the asymptotic behavior of the velocity potential near
infinity into questions about the regularity of the transformed
potential near a finite point on the boundary of the transformed fluid
domain. We obtain this boundary regularity using standard Schauder
estimates for weak solutions to elliptic equations. 

This paper is organized as follows. In Section~\ref{sec:results}, we
state our main results. In Section~\ref{sec:dipole}, we prove
Theorem~\ref{thm:dipole} on the asymptotic behavior of the velocity
potential and free surface near infinity. To streamline the
presentation, some of the more technical details are deferred to
Appendix~\ref{sec:appendix}. In Section~\ref{sec:identities}, we prove
Theorem~\ref{thm:kinetic}, which expresses the kinetic energy in terms
of the wave speed and dipole moment, as well as its corollaries. The
proof involves applying the divergence theorem to a carefully chosen
vector field and then using the asymptotics from
Theorem~\ref{thm:dipole} to deal with one of the boundary terms.

\section{Results}\label{sec:results}

There are two distinguished reference frames for a solitary wave: a
moving frame where the motion appears steady, and another ``lab''
frame where the fluid velocity is assumed to vanish at infinity. As is
common practice, we measure positions in the first frame and
velocities in the second. We set $\bx = (\nx,\ny) \in \R^{\sdim-1} \by
\R$, $\sdim=2$ or $3$, with $\nx$ the horizontal coordinate and $\ny$
the vertical coordinate. Assuming that the free surface $\surf$ is a
graph $\ny=\eta(\nx)$, the semi-infinite fluid domain is 
\begin{align*}
  \fluid = \{(\nx,\ny) \in \R^\sdim : \ny < \eta(\nx)\}.
\end{align*}
Since the fluid is irrotational, the fluid velocity in the lab frame
is the gradient of a velocity potential $\varphi$. The equations satisfied
by $\varphi$ and $\eta$ are 
\begin{subequations}\label{eqn:ww}
  \begin{alignat}{2}
    \label{eqn:ww:laplace}
    \Delta \varphi &= 0 &\quad& \ina \fluid, \\
    \label{eqn:ww:kinematic}
    \grad \varphi \cdot \bn &= \bc \cdot \bn && \ona \surf,\\
    \label{eqn:ww:dynamic}
    \tfrac 12\abs{\grad \varphi}^2 - \bc \cdot \grad \varphi  + g
    \eta &= -\sigma \grad \cdot \bn && \ona \surf, \\
    \label{eqn:ww:etadecay}
    \eta &\to 0 &&\asa \abs \nx \to \infty, \\
    \label{eqn:ww:phidecay}
    \varphi,\grad\varphi &\to 0 && \asa \abs \bx \to \infty,
  \end{alignat}
\end{subequations}
where here $g > 0$ is the constant acceleration due to gravity,
$\sigma \ge 0$ is the constant coefficient of surface tension, $\bc =
(\bcx,0) \in \R^\sdim$ is the (nonzero) wavespeed, and $\bn=\bn(\nx)$
is the unit normal vector to $\surf$ pointing out of $\fluid$.

\begin{subequations}\label{eqn:epsdecay}
  Our main assumptions are that $\varphi$ satisfies 
  \begin{align}
    \label{eqn:epsdecay:phi}
    \varphi 
    = o\bigg( \frac 1{\abs \bx^{\sdim-2}} \bigg) 
    \quad \asa \abs \bx \to \infty,
  \end{align}
  while $\eta$ and its derivatives satisfy
  \begin{align}
    \label{eqn:epsdecay:eta}
    \eta 
    = O\bigg( \frac 1{\abs \nx^{\sdim-1+\varepsilon}}\bigg),
    \quad
    \frac{\partial \eta}{\partial \nx_i} 
    = O\bigg( \frac 1{\abs \nx^{\sdim+\varepsilon}}\bigg),
    \quad 
    \frac{\partial^2 \eta}{\partial \nx_i \partial \nx_j} 
    = O\bigg( \frac 1{\abs \nx^{\sdim+1+\varepsilon}}\bigg),
  \end{align}
  as $\abs \nx \to \infty$ for some $\varepsilon \in (0,1)$ and all
  $i,j$. Note that \eqref{eqn:epsdecay:phi} follows from
  \eqref{eqn:ww:phidecay} when $\sdim = 2$.
\end{subequations}

Our first result is that $\varphi$ behaves like a dipole at infinity.
\begin{theorem}\label{thm:dipole}
  Let $\eta \in C^2(\R^{\sdim-1})$ and $\varphi \in
  C^2(\overline{\fluid})$ solve \eqref{eqn:ww} with $\sigma \ge 0$ and
  suppose that the decay estimates \eqref{eqn:epsdecay} hold. Then
  there exists a ``dipole moment'' $\ba = (\bax,0) \in \R^\sdim$ such
  that $\varphi$ satisfies 
  \begin{align}
    \label{eqn:dipole}
    \varphi = \frac{\ba \cdot \bx}{\abs \bx^\sdim} 
    + O\bigg( \frac 1{\abs \bx ^{\sdim-1+\varepsilon}}\bigg),
    \enskip
    \grad \varphi = \grad \frac{\ba \cdot \bx}{\abs \bx^\sdim} 
    + O\bigg( \frac 1{\abs \bx^{\sdim+\varepsilon}} \bigg)
    \enskip \asa \abs\bx \to \infty
  \end{align}
  while $\eta$ satisfies
  \begin{align}
    \label{eqn:asym}
    \eta = \frac 1{g\abs\nx^\sdim} \left( \bc \cdot \ba
    - \sdim \frac{(\bc \cdot \nx)(\ba \cdot \nx)}{\abs \nx^2} \right)
    + O\bigg( \frac 1{\abs\nx^{\sdim+\varepsilon}} \bigg)
    \quad \asa \abs\nx \to \infty.
  \end{align}
\end{theorem}
Dipole asymptotics along the lines of \eqref{eqn:dipole} feature in
Longuet-Higgins's numerical calculations of two-dimensional
gravity-capillary waves \cite[Section~6]{lh:cgdeep} as well as
Benjamin and Olver's discussion of conserved quantities in the
two-dimensional time-dependent problem
in~\cite[Section~6.5]{bo:symmetry}. Comparing to small-amplitude
expansions, the asymptotic formula \eqref{eqn:asym} is consistent with
(4.1) in \cite{adg:mean} in two dimensions and (5.6) in
\cite{ka:lumps} in three dimensions when $\ba$ and $\bc$ are parallel.

For two-dimensional gravity-capillary waves, Sun proves, roughly, that
the decay $\eta = O(\abs\nx^{-1-\varepsilon})$ forces $\eta =
O(\abs\nx^{-2})$ \cite{sun:properties}. Our result is stronger in that
it identifies the leading order term in the asymptotics, but weaker in
that it requires slightly more information about derivatives of $\eta$
(Sun works in weighted $C^{1+\alpha}$ spaces). Sun also allows for a
semi-infinite upper layer with a different density, and, in the
important special case when $\sigma > \abs\bc^2/4g$, only needs to
assume $\eta = O(\abs\nx^{-\varepsilon})$. As mentioned at the end of
Section~\ref{sec:intro}, his proof uses conformal mappings and hence
does not generalize to three dimensions.

Our next result is an integral identity involving the dipole moment
$\ba$ from Theorem~\ref{thm:dipole}.
\begin{theorem}\label{thm:kinetic}
  In the setting of Theorem~\ref{thm:dipole}, the kinetic energy,
  dipole moment $\ba$, and wave speed $\bc$ are related by
  \begin{align}
    \label{eqn:kinetic}
    \frac 12
    \int_\fluid \abs{\grad \varphi}^2\, d\bx 
    &= 
    - \frac{\pi^{\sdim/2}}{2\Gamma(\frac \sdim 2)}  (\bc \cdot \ba).
  \end{align}
\end{theorem}
For two-dimensional gravity-capillary waves, Longuet-Higgins
\cite{lh:cgdeep} gave a formal proof of \eqref{eqn:kinetic} assuming
\eqref{eqn:dipole}; also see equation (6.22) in \cite{bo:symmetry}.
For three-dimensional waves, however, \eqref{eqn:kinetic} seems to be
new. 

Theorem~\ref{thm:kinetic} implies $\bc \cdot \ba < 0$ for nontrivial
waves with $\varphi \not \equiv 0$, and hence that the leading-order
terms in \eqref{eqn:dipole} and \eqref{eqn:asym} do not vanish. For
$\sdim = 2$, solving \eqref{eqn:kinetic} for $\ba$ and substituting
into \eqref{eqn:asym} yields
\begin{align*}
  \eta = 
  \left(\frac \pi g
  \int_\fluid \abs{\grad\varphi}^2\, d\bx\right)
  \frac 1{\abs\nx^2} + O\bigg( \frac 1{\abs\nx^{2+\varepsilon}}
  \bigg),
\end{align*}
which for instance implies that $\eta > 0$ for $\abs\nx$ sufficiently
large. For $\sdim = 3$, $\eta$ instead takes both positive and
negative values in any neighborhood of infinity.

Another consequence of Theorem~\ref{thm:kinetic} is the following
dichotomy.
\begin{corollary}\label{cor:angular}
  In the setting of Theorem~\ref{thm:dipole}, either
  the angular momentum $\int_\fluid \bx \by \grad \varphi\, d\bx$ is
  infinite or the wave is trivial, i.e.~$\varphi \equiv 0$ and $\eta
  \equiv 0$. 
\end{corollary}
For two-dimensional time-dependent waves, Benjamin and Olver observe
that ${\bax \ne 0}$ causes the integral defining the total angular
momentum to diverge~\cite[Section~6.5]{bo:symmetry}. This motivates
them to impose additional restrictions guaranteeing $\bax = 0$, but
not necessarily $\bay = 0$. For two-dimensional gravity-capillary
waves, Longuet-Higgins \cite{lh:cgdeep} explains how dipole behavior
for the velocity potential causes the horizontal momentum to be
indeterminant in that, for instance, $\grad\varphi\nin L^1(\fluid)$.

A final corollary of the proof of Theorem~\ref{thm:kinetic} is that
the ``excess mass'' vanishes.
\begin{corollary}\label{cor:mass}
  In the setting of Theorem~\ref{thm:dipole}, the wave has zero
  excess mass in that
  \begin{align}
    \label{eqn:mass}
    \int_{\R^{\sdim-1}} \eta(\nx)\, d\nx = 0.
  \end{align}
\end{corollary}
For two-dimensional capillary-gravity waves, \eqref{eqn:mass} was
derived in \cite{lh:cgdeep}, and a stronger version of
Corollary~\ref{cor:mass} was proved rigorously in
\cite{sun:properties}. For three-dimensional waves, however,
Corollary~\ref{cor:mass} appears to be new. An obvious consequence is
that no nontrivial waves satisfying \eqref{eqn:epsdecay} are waves of
elevation (with $\eta > 0$) or depression (with $\eta < 0$). In this
setting, however, waves of elevation and depression have been ruled
out by Craig \cite{craig:nonexistence} using maximum principle
arguments, without imposing any assumptions on the decay rates of
$\eta$ and $\varphi$.

\section{Proof of Theorem~\ref{thm:dipole}}\label{sec:dipole}
The main ingredient in the proof of Theorem~\ref{thm:dipole} is the
following lemma, which states that \eqref{eqn:dipole} holds
independently of the dynamic boundary condition \eqref{eqn:ww:dynamic}. 
\begin{lemma}\label{lem:dipole}
  Let $\varphi \in C^2(\overline\fluid)$ and $\eta \in
  C^2(\R^{\sdim-1})$ solve
  \eqref{eqn:ww:laplace}--\eqref{eqn:ww:kinematic}. If the decay
  estimates \eqref{eqn:epsdecay} hold, then there exists $\ba =
  (\bax,0) \in
  \R^\sdim$ (possibly zero) so that $\varphi$ satisfies the asymptotic
  conditions \eqref{eqn:dipole}.
  \begin{proof}
    We apply the Kelvin transform, setting
    \begin{align*}
      \kbx = T(x) = \frac \bx{\abs\bx^2},
      \qquad 
      \kphi(\kbx) = \frac 1{\abs\kbx^{\sdim-2}}
      \varphi\bigg( \frac{\kbx}{\abs\kbx^2} \bigg),
      \qquad 
      \kfluid = T(\fluid\without B_1),
    \end{align*}
    where $B_1 = \{ \bx : \abs \bx < 1\}$ is an open ball centered at
    the origin. Note that $T(T(x)) = x$. This change of variables
    converts asymptotic questions about $\varphi$ as $\abs\bx \to
    \infty$ into local questions about $\kphi$ in a neighborhood of $0
    \in \dell\kfluid$. For instance, \eqref{eqn:epsdecay:phi} implies
    that $\kphi$ extends to a $C^0(\overline\kfluid)$ function with
    $\kphi(0) = 0$.

    Using the decay assumptions \eqref{eqn:epsdecay}, we show in the
    appendix that $\kfluid$ has a $C^2$ boundary portion $\ksurf$
    containing $0$ and that $\kphi \in H^1(\kfluid)$ is a weak
    solution to the boundary-value problem
    \begin{align}
      \label{eqn:weak}
      \Delta \kphi = 0 \ina \kfluid,
      \qquad 
      \pp \kphi \kbn + \alpha \kphi = \source \ona \ksurf,
    \end{align}
    where $\alpha,\source \in C^\varepsilon(\ksurf)$ are given up to
    a sign by
    \begin{align*}
      \alpha(\kbx) =  -(\sdim-2) \big(\bx \cdot \bn(\bx)\big),
      \quad
      \source(\kbx) = \abs \bx^\sdim \big(\bc \cdot \bn(\bx)\big)
      \qquad 
      \ona \ksurf \without \{0\}
    \end{align*}
    and $\alpha(0) = \source(0) = 0$.
    Standard elliptic regularity theory (for instance Theorem~5.51 in
    \cite{lieberman:oblique}) then implies that $\kphi \in
    C^{1+\varepsilon}(\kfluid \cup \ksurf)$. In particular, setting
    $\ba = (\bax, \bay) = \grad \kphi(0)$, we have an
    expansion
    \begin{align}
      \label{eqn:kasym}
      \kphi(\kbx) = \ba \cdot \kbx + O(\abs\kbx^{1+\varepsilon}),
      \qquad 
      \grad \kphi(\kbx) = \ba  + O(\abs\kbx^\varepsilon)
    \end{align}
    as $\kbx \to 0$. Rewriting \eqref{eqn:kasym} in terms of $\varphi$
    and $\grad\varphi$ then yields \eqref{eqn:dipole} as desired.
    Finally, plugging $\kbx = 0$ in the boundary condition in
    \eqref{eqn:weak}, we find
    \begin{gather*}
      \bay = \pp\kphi\kbn(0) = -\alpha(0)\kphi(0) + \source(0) = 0.
      \qedhere
    \end{gather*}
  \end{proof}
\end{lemma}

Theorem~\ref{thm:dipole} now follows from Lemma~\ref{lem:dipole} and
the dynamic boundary condition \eqref{eqn:ww:dynamic}.
\begin{proof}[Proof of Theorem~\ref{thm:dipole}]
  We have already shown \eqref{eqn:dipole}, so it suffices to prove
  \eqref{eqn:asym}. From \eqref{eqn:epsdecay:eta},
  we have $\grad \cdot \bn = O(1/\abs\nx^{\sdim+1+\varepsilon})$. 
  Solving the dynamic boundary condition \eqref{eqn:ww:dynamic} for
  $\eta$ and plugging in \eqref{eqn:dipole} therefore yields
  \begin{align}
    \label{eqn:replaceme}
    \eta(\nx)
    &= 
    \frac 1{g\abs\bx^\sdim} \left( \bc \cdot \ba
    - \sdim \frac{(\bc \cdot \nx)(\ba \cdot \nx)}{\abs \bx^2} \right)
    + O\bigg( \frac 1{\abs\bx^{\sdim+\varepsilon}} \bigg)
    + O\bigg( \frac 1{\abs\nx^{\sdim+1+\varepsilon}} \bigg),
  \end{align}
  where here $\bx$ is shorthand for $(\nx,\eta(\nx))$. Since $\eta \to
  0$ as $\abs\nx \to \infty$, we can replace each occurrence of $\bx$ in
  \eqref{eqn:replaceme} with $(\nx,0)$, yielding \eqref{eqn:asym} as
  desired.
\end{proof}

\section{Integral identities}\label{sec:identities}

Let $B_r = \{ \bx : \abs\bx < r \}$ denote the open ball
with radius $r$ centered at the origin, and let $\ey = (0,1)$ be the
unit vector in the vertical direction.
\begin{proof}[Proof of Theorem~\ref{thm:kinetic}]
  Consider the vector field
  \begin{align*}
    \bA
    &:= 
    \bigg(-\frac {\abs\bc^2}g (\ey \cdot \grad \varphi )
    + \bc \cdot \bx + \varphi \bigg) 
    \grad \varphi \\
    &\qquad + \frac{\abs\bc^2}g \bigg( \frac 12 \abs{\grad \varphi}^2
    - \bc \cdot \grad \varphi \bigg) 
    \ey 
    +
    \bigg( \frac{\abs\bc^2}g ( \ey \cdot \grad \varphi )
    - \varphi \bigg)
    \bc.
  \end{align*}
  A simple calculation using only the fact that $\varphi$ is harmonic
  shows that $\grad \cdot \bA = \abs{\grad \varphi}^2$. 
  Thus we can apply the divergence theorem to
  $\bA$ on the bounded region $B_r \cap \fluid$ to obtain
  \begin{align}
    \label{eqn:divA}
    \int_{B_r \cap \fluid} \abs{\grad \varphi}^2\, d \bx 
    &=
    \int_{B_r \cap S} \bA \cdot \bn \, \dS
    +\int_{\dell B_r \cap \fluid} \bA \cdot \bn \, \dS.
  \end{align}
  Note that $B_r \cap S$ is the portion of the boundary of $B_r \cap
  \fluid$ on the free surface while $\dell B_r \cap \fluid$ is the
  portion inside the fluid. 

  On the free surface $\surf$, we have
  \begin{align}
    \label{eqn:ndS}
    \bn =  \frac{(-\grad \eta,1)}{\sqrt{1+\abs{\grad\eta}^2}},
    \qquad 
    \dS = \sqrt{1+\abs{\grad\eta}^2}\, d\nx,
  \end{align}
  while the boundary conditions \eqref{eqn:ww:kinematic} and
  \eqref{eqn:ww:dynamic} imply
  \begin{align*}
    \bA \cdot \bn 
    &= 
    (\bc \cdot \bx)
    (\bc \cdot \bn)
    - \left(\frac {\abs\bc^2\sigma}{g} \grad \cdot \bn + \abs\bc^2 \eta \right)
    (\ey \cdot \bn).
  \end{align*}
  Thus the first term on the right hand side of \eqref{eqn:divA} can
  be rewritten as
  \begin{align}
    \notag
    \int_{B_r \cap \surf} \bA \cdot \bn\, \dS
    &= 
    -\int_{B_r \cap \surf} \left((\bc \cdot \bx)
    (\bc \cdot \grad \eta)
    +\frac {\abs\bc^2\sigma}{g} \grad \cdot \bn
    + \abs\bc^2 \eta\right) d\nx\\
    \notag
    &=  -\int_{B_r \cap \surf} \grad \cdot \left(    
    \frac {\abs\bc^2\sigma}{g}  \bn +
    \eta (\bc \cdot \bx) \bc \right)\, d\nx\\
    \label{eqn:divA1}
    &=  \int_{\dell B_r \cap \surf} \left(    
    \frac {\abs\bc^2\sigma}{g}  \bn +
    \eta (\bc \cdot \bx) \bc \right) \cdot \bnu\, ds,
  \end{align}
  where here the outward-pointing normal $\bnu \maps T \to \R^{\sdim-1}$ 
  and measure $ds$ are with respect to the projection of $\dell B_r
  \cap \surf$ onto $\R^{\sdim-1}$.

  Plugging \eqref{eqn:divA1} into \eqref{eqn:divA} we obtain
  \begin{align}
    \label{eqn:killrhs}
    \begin{aligned}
      \int_{\fluid \cap B_r} \abs{\grad \varphi}^2\, d \bx 
      &=
      \frac {\abs\bc^2\sigma}{g} \int_{\surf \cap \dell B_r} \bn \cdot \bnu\, ds
      -\int_{\surf \cap \dell B_r} \eta (\bc \cdot \bx) (\bc \cdot \bnu) \, ds\\
      &\qquad\qquad
      +\int_{\fluid \cap \dell B_r} \bA \cdot \bn \, \dS.
    \end{aligned}
  \end{align}
  From \eqref{eqn:ndS} and \eqref{eqn:epsdecay:eta} we see that the first
  integrand on the right hand side of \eqref{eqn:killrhs} is
  $O(\abs{\grad\eta}) = O(\abs\nx^{-(\sdim+\varepsilon)})$ while the
  second integrand is $O(\abs{\nx}^{-(\sdim+2+\varepsilon)})$. Thus these
  first two integrals vanish as $r \to \infty$.
  Thanks to the asymptotic conditions \eqref{eqn:dipole} proved in
  Theorem~\ref{thm:dipole}, the remaining integral converges, as $r \to
  \infty$, to the constant value
  \begin{align*}
    \int_{\dell B_r \cap \{\ny<0\}} 
    \left(      (\bc \cdot \bx) \grad \frac{\ba \cdot \bx}{\abs \bx^\sdim}
    - \frac{ \ba \cdot \bx }{\abs \bx^\sdim} \bc 
    \right) \cdot 
    \frac{\bx}{\abs \bx} \, \dS
    &= 
    -\sdim
    \int_{\dell B_1 \cap \{\ny<0\}} (\bc \cdot \bx)(\ba \cdot \bx) \, \dS\\
    &= - \frac{\pi^{\sdim/2}}{\Gamma(\frac \sdim 2)}  (\bc \cdot \ba),
  \end{align*}
  leaving us with \eqref{eqn:kinetic} as desired.
\end{proof}

\begin{proof}[Proof of Corollary~\ref{cor:mass}]
  We follow the proof of Theorem~\ref{thm:kinetic}, but with 
  $\bA$ replaced by the vector field
  \begin{align*}
    \bC &= -\frac {\abs\bc^2}g (\ey \cdot \grad \varphi )
    \grad \varphi 
    + \frac{\abs\bc^2}g \left( \frac 12 \abs{\grad \varphi}^2
    - \bc \cdot \grad \varphi \right) 
    \ey 
    + \frac{\abs\bc^2}g ( \ey \cdot \grad \varphi )
    \bc
  \end{align*}
  obtained by dropping all of the terms in $\bA$ without a factor of
  $\abs\bc^2/g$. A simple calculation shows that $\grad \cdot \bC =
  0$, and the boundary conditions \eqref{eqn:ww:kinematic} and
  \eqref{eqn:ww:dynamic} give
  \begin{align*}
    \bC \cdot \bn =  - \abs\bc^2 \eta - \frac{\abs\bc^2\sigma}g \grad
    \cdot \bn
  \end{align*}
  on the free surface $\surf$. As in the proof of
  Theorem~\ref{thm:kinetic}, we apply the divergence theorem, first to
  $\bC$ in $B_r \cap \fluid$, and then again on $\surf \cap B_r$,
  obtaining
  \begin{align}
    \label{eqn:firstterm}
    \abs\bc^2 \int_{ B_r \cap \surf} \eta\, d\nx
    &= 
    -\frac {\abs\bc^2\sigma}{g} \int_{\surf \cap \dell B_r} \bn \cdot \bnu\, ds
    +\int_{\fluid \cap \dell B_r} \bC \cdot \bn \, \dS.
  \end{align}
  The first term on the right hand side of \eqref{eqn:firstterm}
  vanishes as $r \to \infty$ as in proof of
  Theorem~\ref{thm:kinetic}. The second term vanishes since
  $\tilde\bA = O(\abs{\grad\varphi}) = O(\abs x^{-\sdim})$ by
  \eqref{eqn:dipole}. From \eqref{eqn:epsdecay:eta} we know that
  $\eta \in L^1(\R^{\sdim-1})$, so taking $r \to \infty$ in
  \eqref{eqn:firstterm} yields \eqref{eqn:mass} as desired.
\end{proof}

\begin{proof}[Proof of Corollary~\ref{cor:angular}]
  For any $r > 0$, the asymptotic condition \eqref{eqn:dipole}
  implies that the integral
  \begin{align*}
    \int_{\fluid \cap \dell B_r} \bx \by \grad \varphi \, d S
  \end{align*}
  converges, as $r \to \infty$, to the constant value
  \begin{align}
    \label{eqn:mustbe}
    \int_{\dell B_r \cap \{\ny < 0\}} 
    \bx \by \grad \frac{\ba \cdot \bx}{\abs \bx^\sdim} \, d S
    = - \ba \by \int_{\dell B_1 \cap \{ \ny < 0\}}
    \bx \, d S
    = \frac{\pi^{(\sdim-1)/2}}{\Gamma\left(\frac{\sdim+1}2\right)} \ba \by \ey.
  \end{align}
  Suppose that the angular momentum is finite. Then the right hand
  side of \eqref{eqn:mustbe} must be zero, which forces $\ba \by \ey =
  0$ and hence $\bax = 0$. But then $\bc \cdot
  \ba = 0$ so that \eqref{eqn:kinetic} gives $\grad\varphi \equiv 0$
  and therefore $\varphi \equiv 0$.

  It remains to show $\eta \equiv 0$. Plugging $\varphi \equiv 0$ 
  into \eqref{eqn:ww:dynamic}, we find $g\eta = -\sigma \grad \cdot
  \bn$. At a positive maximum of $\eta$, this reduces to $0 < g\eta =
  \sigma \Delta \eta \le 0$, a contradiction. Similarly at a negative
  minimum of $\eta$ we have $0 > g\eta = \sigma \Delta \eta \ge 0$,
  again a contradiction, and we conclude that $\eta \equiv 0$.
\end{proof}

\appendix
\section{}\label{sec:appendix}

In this appendix we provide the remaining details in the proof of
Lemma~\ref{lem:dipole}.

Setting $\ksurf = (B_\delta \cap T(\surf)) \cup \{0\}\sub \dell\kfluid$ for
$\delta$ sufficiently small, we first claim that $\ksurf$ is a $C^2$
graph $\kny = f(\knx)$, which will imply that $\ksurf$ is a $C^2$
boundary portion. As an intermediate step, we define yet another
variable
\begin{align*}
  \kax = \frac \nx{\abs\nx^2} 
\end{align*}
which, on $\ksurf \without \{0\}$, is related to $\knx$ via
\begin{align}
  \label{eqn:kaxtime}
  \knx 
  = \frac{\kax}{1 + \abs\kax^2\eta^2(\kax/\abs\kax^2)}.
\end{align}
Our decay assumptions \eqref{eqn:epsdecay:eta} easily imply that
$\abs\kax^2 \eta^2(\kax/\abs\kax^2)$ extends to a $C^2$ function of
$\kax$ in a neighborhood of $\kax = 0$ which vanishes at $\kax = 0$
together with its first and second derivatives. Thus we can use the
implicit function theorem to solve \eqref{eqn:kaxtime} for $\kax$ as a
$C^2$ function of $\knx$. Expressing $\kny$ in terms of $\kax$,
\begin{align}
  \label{eqn:knytime}
  \kny 
  = \frac{\abs\kax^2\eta(\kax/\abs\kax^2)}
  {1 + \abs\kax^2\eta^2(\kax/\abs\kax^2)},
\end{align}
\eqref{eqn:epsdecay:eta} similarly guarantees that $\kny$ can be
extended to a $C^2$ function of $\kax$ in a neighborhood of $\kax =
0$. Composing the $C^2$ mappings $\kax \mapsto \kny$ and $\knx \mapsto
\kax$ yields the desired equation $\kny = f(\knx)$ for $\ksurf$.

We now consider the boundary condition satisfied by $\kphi$ on
$\ksurf$. A calculation shows that a $C^1$ unit
normal $\kbn$ on $\ksurf$ is given in terms of the normal vector $\bn$ on
$\surf$ via the formula
\begin{align*}
  \kbn(\kbx)
  = \bn(\bx) - 2 \frac{\bn(\bx)\cdot\bx}{\abs\bx^2}.
\end{align*}
For simplicity assume that $\kbn$ points out of $\kfluid$;
otherwise the definitions of $\source$ and $\alpha$ below are off by
an unimportant sign. Differentiating the identity
\begin{align*}
  \varphi(\bx) = \frac 1{\abs\bx^{\sdim-2}} 
  \kphi\bigg( \frac \bx{\abs\bx^2} \bigg)
\end{align*}
and using the boundary condition \eqref{eqn:ww:kinematic}, we find
that, on $\ksurf \without \{0\}$,
\begin{align*}
  \bc \cdot \bn 
  &= \pp \varphi \bn
  = -(\sdim-2)\frac {\bx \cdot \bn}{\abs \bx^\sdim} \kphi
  + \frac 1{\abs \bx^\sdim} \pp \kphi\kbn.
\end{align*}
Multiplying through by $\abs\bx^\sdim$, we write this as
\begin{align*}
  \pp \kphi \kbn
  + \alpha \kphi
  = \source
  \qquad \ona \ksurf 
  \without \{0\},
\end{align*}
where $\alpha$ and $\source$ are defined on $\ksurf \without \{0\}$ by
\begin{align*}
  \alpha(\kbx) =  -(\sdim-2) (\bx \cdot \bn),
  \qquad 
  \source(\kbx) = \abs \bx^\sdim (\bc \cdot \bn).
\end{align*}
Clearly $\alpha, \source \in C^1(\ksurf \without \{0\})$. Using our
decay assumptions \eqref{eqn:epsdecay:eta}, we check that they extend
to $C^\varepsilon$ functions of $\knx$ vanishing at $\knx = 0$. We
remark that only this last extension of $\source$ requires the full
force of \eqref{eqn:epsdecay:eta}; the other extensions only require
$\eta = O(1/\abs\bx^\varepsilon)$, $D \eta =
O(1/\abs\bx^{1+\varepsilon})$, and $D^2\eta =
O(1/\abs\bx^{2+\varepsilon})$.

Next we show that $\kphi \in H^1(\kfluid)$. From
\eqref{eqn:epsdecay:eta} we know $\kphi \in C^0(\overline\kfluid) \cap
C^2(\overline\kfluid\without \{0\})$, so it is enough to show $\grad
\kphi \in L^2(\kfluid \cap B_R)$ for some $R>0$. Fix $R$ small enough
that $B_{2R} \cap \dell\kfluid \sub \ksurf$, let $\eta_0 \in
C^\infty_\mathrm{c}(\R)$ be a nonnegative function satisfying
$\eta_0(s) = 0$ for $s < 1$ and $\eta_0(s) = 1$ for $s > 2$, and for $r <
R/2$ define
\begin{align*}
  \eta(\kbx;r) 
  = 
  \eta_0(r\inv\abs\kbx) \Big(1-\eta_0\big(R\inv \abs\kbx\big)\Big).
\end{align*}
Multiplying $\Delta\kphi = 0$ by $\eta^2 \kphi$ and integrating by
parts, we find
\begin{align}
  \notag
  \int_{\kfluid} \abs{\eta \grad \kphi}^2 \, d\kbx
  &=  
  - 2\int_{\kfluid} 
  \eta \kphi\grad \eta \cdot \grad \kphi \, d\kbx
  + \int_{\ksurf} \eta \kphi ( \source - \alpha \kphi )\, \dS\\
  \label{eqn:gpl2}
  &\le C\left( 1 + 
  \n{\grad \eta}_{L^2(\R^\sdim)}
   \n{\eta \grad \kphi}_{L^2(\kfluid)}
  \right),
\end{align}
where $C$ depends on the $L^\infty$ norms of $\kphi,\alpha,\source$.
Since
\begin{align*}
  \n{\grad \eta}_{L^2(\R^\sdim)} \le C(1+r^{\sdim-2}) \le C,
\end{align*}
\eqref{eqn:gpl2} implies an upper bound on $\n{\eta \grad
\kphi}_{L^2(\kfluid)}$ independent of $r < R/2$. Sending $r \to 0$, we
obtain $\grad\kphi \in L^2(\kfluid \cap B_R)$ as desired.

Finally, we claim that $\kphi$ is a weak solution to
\eqref{eqn:weak}. Certainly (after perhaps changing the
definitions of $\alpha,\source$ by a sign)
\begin{align}
  \label{eqn:weakform}
  \int_{\kfluid} \grad \kphi \cdot \grad v\, d\kbx
  = \int_{\ksurf} (\alpha \kphi - \source)v\, d\kbx
\end{align}
for all smooth $v \in H^1(\kfluid)$ vanishing in a neighborhood of
$0$. Such $v$ are dense in $H^1(\kfluid)$ (see, for instance,
Lemmas~17.2 and 17.3 in \cite{tartar:book}). Since $\kphi \in
H^1(\kfluid)$, \eqref{eqn:weakform} therefore holds for all $v \in
H^1(\kfluid)$ and the claim is proved.

\medbreak
\textbf{Acknowledgments.} 
The author thanks Dennis Kriventsov for many helpful discussions on
the proof of Lemma~\ref{lem:dipole}.
This research was supported by the National Science Foundation under
Award No.~DMS-1400926.

\end{document}